\documentclass[twoside, a4paper,12pt]{amsart}

\usepackage{tikz}
\usepackage{amsfonts}
\usepackage{amssymb}
\usepackage{array}
\newcolumntype{P}[1]{>{\centering\arraybackslash}p{#1}}
\usepackage[hmarginratio=1:1]{geometry}

\theoremstyle{plain}
\newtheorem*{claim*}{Claim}
\newtheorem{thm}{Theorem}[section]
\newtheorem{corollary}[thm]{Corollary}
\newtheorem{lemma}[thm]{Lemma}
\newtheorem{prop}[thm]{Proposition}
\theoremstyle{definition}
\newtheorem{defn}[thm]{Definition}
\newtheorem{ex}[thm]{Example}

\newtheorem{remark}[thm]{Remark}

\newtheorem{prob}[thm]{Open Problem}
\begin{document}
\title{\large{Semigroups for which every right congruence of finite index is finitely generated}}
\author{Craig Miller}
\address{School of Mathematics and Statistics, St Andrews, Scotland, UK, KY16 NSS}
\email{cm380@st-andrews.ac.uk}

\begin{abstract}
We call a semigroup $S$ {\em f-noetherian} if every right congruence of finite index on $S$ is finitely generated.
We prove that every finitely generated semigroup is f-noetherian, 
and investigate whether the properties of being f-noetherian and being finitely generated coincide for various semigroup classes.
\end{abstract}

\maketitle 
\section{\large{Introduction}\nopunct}

A {\em finiteness condition} for a class of universal algebras is a property that is satisfied by at least all finite members of that class.
The study of algebras via their finiteness conditions has been of significant importance in understanding the structure and behaviour of various kinds of algebraic structures.
In recent decades there has been a considerable amount of interest in finiteness conditions on semigroups concerning their lattices of one-sided congruences.
One-sided congruences play a key role in the theory of monoid acts (or, equivalently, representations of monoids by transformations of sets).  
Indeed, the cyclic right (resp. left) acts of a monoid $M$ correspond to the right (resp. left) congruences on $M.$
This means that finiteness conditions on monoids relating to their actions usually have equivalent formulations in terms of one-sided congruences; 
for example, the notions of {\em coherency} \cite{Gould} and being {\em noetherian} \cite{Normak}.\par
Recently, in \cite{Dandan}, Dandan, Gould, Quinn-Gregson and Zenab considered the property that the universal right congruence on a semigroup is finitely generated; 
we refer to such semigroups as {\em finitely connected} (see Remark \ref{fc remark} for the motivation behind using this term).
Semigroups for which every right congruence is finitely generated, called {\em right noetherian} semigroups, 
have received a fair amount of attention (see \cite{Hotzel, Kozhukhov1, Kozhukhov2, Miller3}).
Both of these properties have an interesting relationship to finite generation.
It was proved in \cite{Dandan} that every finitely generated semigroup is finitely connected.
The authors also prove there that finitely connected completely simple semigroups are finitely generated \cite[Cor. 3.6]{Dandan}.
However, an important strand of that paper is demostrating that, for general semigroups, the property of being finitely connected is substantially weaker than being finitely generated.
For example, every monoid with a zero is finitely connected \cite[Corollary 2.15]{Dandan}.\par
The situation is rather different for the property of being right noetherian.
It is easy to find finitely generated semigroups that are not right noetherian; e.g. the free monoid on two generators.
It is well known that a commutative semigroup is (right) noetherian if and only if it is finitely generated.
An intriguing open problem is whether every right noetherian semigroup $S$ is finitely generated.
This question was originally asked by Hotzel in \cite{Hotzel}, and has been considered in several subsequent papers \cite{Bergman, Kozhukhov2, Miller3}.
It has been answered in the affirmative in the following cases:
if $S$ is weakly periodic (which inludes semisimple, and hence regular, semigroups) \cite[Theorem 3.2]{Hotzel};
if every $\mathcal{J}$-class of $S$ is a class of some congruence on $S$ \cite[Theorem 3.3]{Hotzel};
if $S$ is also left noetherian \cite[Theorem 5]{Kozhukhov2}.\par
In this paper we consider the property that every right congruence of finite index on a semigroup is finitely generated, which we call being {\em f-noetherian}.
We will focus on its relationship with finite generation.
Our condition of being f-noetherian is clearly `between' the properties of being finitely connected and being right noetherian:
right noetherian semigroups are f-noetherian, and f-noetherian semigroups are finitely connected.
It will be shown that every finitely generated semigroup is f-noetherian.
The main purpose of the article is to investigate, for various semigroup classes, whether the property of being f-noetherian coincides with finite generation,
and to see how the situation differs for each of the properties of being finitely connected and being right noetherian.
Specifically, we consider the following classes of semigroups: 
inverse, completely regular, idempotent, commutative, cancellative, nilpotent, archimedean, complete, ($0$-)simple and completely ($0$-)simple.
Our findings are summarised in Table 1.
Particularly noteworthy results are obtained regarding commutative semigroups.  
In particular, we exhibit two examples of f-noetherian commutative semigroups that are not finitely generated, one of which is cancellative, idempotent-free and countable, and the other uncountable.
Furthermore, we prove that being countable is a necessary condition for a cancellative commutative semigroup to be f-noetherian.\par 
The paper is structured as follows.  
In Section 2 we introduce the foundational material required for the rest of the paper.
This section is divided into two subsections.  
In the first of these we make some basic definitions and observations, and provide alternative formulations of the property of being f-noetherian for groups and monoids.
In the second subsection we establish some useful algebraic properties of f-noetherian semigroups.
In Section 3 we prove that certain infinite transformation semigroups are f-noetherian.
We then turn to the main theme of the paper in Section 4.  Here we prove that every finitely generated semigroup is f-noetherian.
In Sections 5-7 we investigate whether the converse holds for semigroups within various `standard' classes.
Section 5 is concerned with completely regular semigroups, Section 6 is devoted to commutative semigroups, and simple and $0$-simple semigroups are considered in Section 7.
Finally, in Section 8, we pose some open problems and directions for future research.

\begin{table}
\footnotesize
\caption{Summary of results\label{long}}
\begin{tabular}{|P{2.5cm}||P{2.5cm}|P{2.5cm}|P{2.5cm}|P{2.5cm}|}
 \hline
 \textbf{Class of semigroup $S$} & \textbf{$S$ finitely connected $\Leftrightarrow S$ f.g.} & \textbf{$S$ f-noetherian $\Leftrightarrow S$ f.g.} & 
 \textbf{$S$ right noetherian $\Rightarrow S$ f.g.} & \textbf{$S$ f.g $\Rightarrow S$ right noetherian}\\
 \hline\hline
 \textbf{Inverse} & No & No & Yes & No\\
  & \cite[Cor. 5.7]{Dandan} & (Example \ref{symmetric inverse}) & \cite[Thm. 3.2]{Hotzel} & (Remark \ref{noetherian group})\\
 \hline
 \textbf{Completely} & No & ??? & Yes & No\\
 \textbf{regular} & \cite[Cor. 5.7]{Dandan} & & \cite[Thm. 3.2]{Hotzel} & (Remark \ref{noetherian group})\\
 \hline
 \textbf{Cryptogroup} & No & Yes & Yes & No\\
  & \cite[Cor. 5.7]{Dandan} & (Thm. \ref{cryptogroup}) & \cite[Thm. 3.2]{Hotzel} & (Remark \ref{noetherian group})\\
 \hline
 \textbf{Band} & No & Yes & Yes & Yes\\
  & \cite[Cor. 5.7]{Dandan} & (Cor. \ref{band}) & \cite[Thm. 3.2]{Hotzel} & (f.g. $\Rightarrow$ finite)\\
 \hline
 \textbf{Commutative} & No & No & Yes & Yes\\
  & \cite[Cor. 5.7]{Dandan} & (Example \ref{commutative ex}) & \cite{Redei} & \cite{Budach}\\
 \hline
 \textbf{Cancellative} & No & No & ??? & No\\
  & (Example \ref{commutative ex}) & (Example \ref{commutative ex}) & & (Remark \ref{noetherian group})\\
 \hline 
 \textbf{Nilpotent} & Yes & Yes & Yes & Yes\\
  & (Lemma \ref{nilpotent}) & (Lemma \ref{nilpotent}) & \cite[Thm. 3.2]{Hotzel} & (f.g. $\Rightarrow$ finite)\\
 \hline
 \textbf{Archimedean} & Yes & Yes & Yes & Yes\\
  & (Prop. \ref{archimedean}) & (Prop. \ref{archimedean}) & \cite{Redei} & \cite{Budach}\\
 \hline
 \textbf{Complete} & No & Yes & Yes & Yes\\
  & \cite[Cor. 5.7]{Dandan} & (Prop. \ref{complete}) & \cite{Redei} & \cite{Budach}\\
 \hline
 \textbf{Left simple} & Yes & Yes & Yes & No\\
  & (Remark \ref{left simple}) & (Remark \ref{left simple}) & \cite[Thm. 3.2]{Hotzel} & (Remark \ref{noetherian group})\\
 \hline
 \textbf{Right simple} & No & No & Yes & No\\
  & (Thm. \ref{Baer-Levi}) & (Thm. \ref{Baer-Levi}) & \cite[Thm. 3.2]{Hotzel} & (Remark \ref{noetherian group})\\
 \hline
 \textbf{Left} & No & Yes & Yes & No\\
 \textbf{$0$-simple} & (Remark \ref{0-simple}(2)) & (Prop. \ref{L-classes}) & \cite[Thm. 3.2]{Hotzel} & (Remark \ref{0-simple}(3))\\
 \hline
 \textbf{Right} & No & No & Yes & No\\
 \textbf{$0$-simple} & (Remark \ref{0-simple}(1)) & (Remark \ref{0-simple}(1)) & \cite[Thm. 3.2]{Hotzel} & (Remark \ref{0-simple}(3))\\
 \hline
 \textbf{Completely} & Yes & Yes & Yes & No\\
 \textbf{simple} & \cite[Cor. 3.6]{Dandan} & (Thm. \ref{completely simple}) & \cite[Thm. 3.2]{Hotzel} & (Remark \ref{noetherian group})\\
 \hline
 \textbf{Completely} & No & Yes & Yes & No\\
 \textbf{$0$-simple} & \cite[Cor. 6.10]{Dandan} & (Thm. \ref{completely 0-simple}) & \cite[Thm. 3.2]{Hotzel} & (Remark \ref{0-simple}(3))\\
 \hline
\end{tabular}
\end{table}

\section{\large{Fundamentals}\nopunct}

\subsection{Definition and equivalent formulations}\hfill\par

The {\em index} of an equivalence relation $\rho$ is the cardinality of the set of equivalence classes of $\rho.$
We call a semigroup $S$ {\em f-noetherian} if every right congruence of finite index on $S$ is finitely generated.\par

Now, let $S$ be a semigroup and let $X\subseteq S\times S$.  We introduce the notation
$$\overline{X}=X\cup\{(x, y) : (y, x)\in X\}.$$
For $a, b\in S$, an {\em $X$-sequence connecting} $a$ and $b$ is any sequence
$$a=x_1s_1, \; y_1s_1=x_2s_2, \; y_2s_2=x_3s_3, \; \dots, \; y_ks_k=b,$$
where $(x_i, y_i)\in\overline{X}$ and $s_i\in S^1$ for $1\leq i\leq k.$\par
We now make the following definition.

\begin{defn}
Let $S$ be a semigroup, let $X\subseteq S\times S,$ and let $s, t\in S.$
We say that $(s, t)$ is a {\em consequence} of $X$ if either $s=t$ or there exists an $X$-sequence connecting $s$ and $t.$
\end{defn}

We have the following basic lemma.

\begin{lemma}
Let $S$ be a semigroup, let $X\subseteq S\times S,$ let $\rho$ be a right congruence on $S$ generated by $X,$ and let $s, t\in S$.
Then $s\,\rho\,t$ if and only if $(s, t)$ is a consequence of $X$.
\end{lemma}

\begin{defn}
Let $S$ be a semigroup.  We say that $S$ is {\em finitely connected} if the universal relation is finitely generated as a right congruence.
\end{defn}

\begin{remark}
\label{fc remark}
The motivation for using the term {\em finitely connected} is \cite[Proposition 3.6]{Dandan}, 
which states that the universal right congruence on a semigroup $S$ is finitely generated if and only if 
there exists a finite set $X\subseteq S$ such that $S=XS^1$ and the undirected left Cayley graph of $S^1$ with respect to $X$ is connected; see \cite[Definition 3.5]{Dandan} for more information.
\end{remark}

It is obvious that every f-noetherian semigroup is finitely connected.  
In fact, the existence of any finitely generated right congruence of finite index implies that the universal right congruence is finitely generated. 

\begin{lemma}
Let $S$ be a semigroup, and let $\rho$ and $\sigma$ be two right congruences of finite index on $S$ with $\rho\subseteq\sigma.$
If $\rho$ is finitely generated, then so is $\sigma.$
\end{lemma}

\begin{proof}
Let $C_i, i\in I,$ be the $\rho$-classes, and let $X$ be a finite generating set for $\rho.$
For each $i\in I,$ fix $\alpha_i\in C_i,$ and define a set 
$$H=\{(\alpha_i, \alpha_j) : i\neq j, \alpha_i,\sigma\,\alpha_j\}.$$
We claim that $\sigma$ is generated by $X\cup H.$
Indeed, let $a\,\sigma\,b$ with $a\neq b.$
There exist $i, j\in I$ such that $a\in C_i$ and $b\in C_j.$
Since $a\,\rho\,\alpha_i$ and $b\,\rho\,\alpha_j,$ we have that $(a, \alpha_i)$ and $(b, \alpha_j)$ are consequences of $X.$
Now, either $i=j$ or $(\alpha_i, \alpha_j)\in H,$ so $(a, b)$ is a consequence of $X\cup H.$
\end{proof}

\begin{corollary}
Let $S$ be a semigroup.  If there exists a right congruence of finite index on $S$ which is finitely generated, then $S$ is finitely connected.
\end{corollary}

For a group $G,$ there is an isomorphism between the lattice of right congruences on $G$ and the lattice of subgroups of $G.$
In particular, right congruences of finite index correspond to subgroups of finite index.
It is well known that in a finitely generated group $G,$ every subgroup of $G$ of finite index is also finitely generated, so we have the following fact.

\begin{lemma}
\label{group}
The following are equivalent for a group $G$:
\begin{enumerate}
\item $G$ is f-noetherian;
\item $G$ is finitely connected;
\item $G$ is finitely generated.
\end{enumerate}
\end{lemma}

\begin{remark}
\label{noetherian group}
For groups, being right noetherian is equivalent to every subgroup being finitely generated \cite[Prop. 2.14]{Miller3}.
There exist finitely generated groups that are not right noetherian; for example, the free group on two generators.
\end{remark}

For a monoid $M,$ there is a one-to-one correspondence between the set of right congruences on $M$ and the set of cyclic $M$-acts.
In the following, we provide some basic definitions about monoid acts; see \cite{Kilp} for more information.\par
Let $M$ be a monoid.  A {\em (right) $M$-act} is a non-empty set $A$ together with a map $$A\times M\to A, (a, m) \mapsto am$$
such that $a(mn)=(am)n$ and $a1=a$ for all $a\in A$ and $m, n\in M.$\par 
For instance, $M$ itself is an $M$-act via right multiplication.\par 
An equivalence relation $\rho$ on an $M$-act $A$ is an {\em ($M$-act) congruence} on $A$ if $(a, b)\in\rho$ implies $(am, bm)\in\rho$ for all $a, b\in A$ and $m\in M.$
Note that the congruences on the $M$-act $M$ are precisely the right congruences on $M.$\par
An $M$-act $A$ is {\em finitely generated} if there exists a finite subset $X\subseteq A$ such that $A=XS^1,$
and $A$ is {\em finitely presented} if it is isomorphic to a quotient of a finitely generated free $M$-act by a finitely generated congruence.
One may consult \cite[Section 1.5]{Kilp} for further details.
A more systematic study of finite presentability of monoid acts was developed in \cite{Miller1, Miller2}.\par
The following result characterises f-noetherian monoids in terms of their acts.

\begin{prop}
The following are equivalent for a monoid $M$:
\begin{enumerate}
 \item $M$ is f-noetherian;
 \item every finite cyclic $M$-act is finitely presented;
 \item every finite $M$-act is finitely presented.
\end{enumerate}
\end{prop}

\begin{proof}
$(1)\Rightarrow (2)$.  Let $A$ be a finite cyclic $M$-act.
By \cite[Proposition 5.17]{Kilp}, we have that $A\cong M/\rho$ where $\rho$ is a right congruence of finite index on $M.$  
Since $M$ is f-noetherian, the right congruence $\rho$ is finitely generated, so $A$ is finitely presented.\par
$(2)\Rightarrow (3)$.  We prove by induction on the number of elements.
The trivial $M$-act $\{0\}$ is finitely presented by assumption.
Now assume that every $M$-act with at most $n$ elements is finitely presented, and let $A$ be an $M$-act with $n+1$ elements.  
If $A$ is the disjoint union of $n+1$ copies of the trivial act, then $A$ is finitely presented by \cite[Corollary 5.9]{Miller2}.
Suppose then that $A$ contains a non-zero element $a,$ and consider the subact $B=\langle a\rangle$ of $A.$  Since $B$ is cyclic, it is finitely presented by assumption.
The Rees quotient $A/B$ is the set $(A\!\setminus\!B)\cup\{0\}$ with action given by
$$a\cdot m=
  \begin{cases} 
   am & \text{if }am\in A\\
   0 & \text{if }am\in B
  \end{cases}$$
and $0\cdot m=0$ for all $a\in A\!\setminus\!B$ and $m\in M.$ 
The number of elements in $A/B$ is $$|A|-|B|+1=(n+2)-|B|\leq(n+2)-2=n,$$ so $A/B$ is finitely presented by the inductive hypothesis.
Since $B$ and $A/B$ are finitely presented, it follows from \cite[Corollary 4.4]{Miller2} that $A$ is finitely presented.\par
$(3)\Rightarrow (1)$.  For any right congruence $\rho$ of finite index on $M,$ we have that the finite cyclic $M$-act $A=M/\rho$ is finitely presented, 
so $\rho$ is finitely generated by \cite[Proposition 3.9]{Miller2}.
\end{proof}

\begin{remark}
It was observed in \cite[Proposition 3.1]{Dandan} that a monoid $M$ being finitely connected is equivalent to the trivial $M$-act being finitely presented.
\end{remark}

\subsection{Algebraic properties}\hfill\par

In the following we establish various algebraic properties of f-noetherian semigroups that will be useful in the remainder of the paper.\par
The following result shows that the property of being f-noetherian is closed under homomorphic images (or equivalently quotients).

\begin{lemma}
\label{quotient}
Let $S$ be a semigroup and let $T$ be a homomorphic image of $S.$ 
If $S$ is f-noetherian, then $T$ is f-noetherian.
\end{lemma}

\begin{proof}
Let $\rho$ be a right congruence of finite index on $T$.
Let $\theta : S\to T$ be a surjective homomorphism, and define a right congruence $\rho^{\prime}$ on $S$ by 
$$s\,\rho^{\prime}\,t\iff s\theta\,\rho\,t\theta.$$
Clearly $\rho^{\prime}$ has findex index.
Since $S$ is f-noetherian, $\rho^{\prime}$ is generated by a finite set $X.$
We claim that $\rho$ is generated by the finite set 
$$Y=\{(x\theta, y\theta) : (x, y)\in X\}.$$
Indeed, let $a\,\rho\,b$ with $a\neq b.$
There exist $s, t\in S$ such that $a=s\theta$ and $b=t\theta.$
Since $s\,\rho^{\prime}\,t$ and $s\neq t,$ there exists an $X$-sequence connecting $s$ and $t.$
Applying $\theta$ to every term of this sequence yields a $Y$-sequence connecting $a$ and $b.$
\end{proof}

\begin{remark}
\label{quotient remark}
Analogues of Lemma \ref{quotient} for the properties of being finitely connected and being right noetherian are proved by essentially the same argument.
\end{remark}

The next result states that semigroups and their subsemigroups of finite complement behave in the same way with regard to being f-noetherian.

\begin{prop}
\label{large}
Let $S$ be a semigroup with a subsemigroup $T$ such that $S\!\setminus\!T$ is finite.  Then $S$ is f-noetherian if and only if $T$ is f-noetherian.
\end{prop}

\begin{proof}
The proof of this result is essentially identical to the prove of the corresponding result for right noetherian semigroups; see \cite[Theorem 3.2]{Miller3}.
\end{proof}

\begin{corollary}
\label{adjoin identity}
A semigroup $S$ is f-noetherian if and only if $S^1$ is f-noetherian.
\end{corollary}

\begin{corollary}
\label{adjoin zero}
A semigroup $S$ is f-noetherian if and only if $S^0$ is f-noetherian.
\end{corollary}

\begin{remark}
An analogue of Proposition \ref{large} holds for the property of being right noetherian.
However, it does not hold for the property of being finitely connected.  
Indeed, adjoining a zero to any monoid yields a finitely connected monoid.
\end{remark}

The next situation we consider is where the complement of a subsemigroup is an ideal.

\begin{lemma}
\label{complement ideal}
Let $S$ be a semigroup with a subsemigroup $T$ such that $S\!\setminus\!T$ is an ideal of $S.$  If $S$ is f-noetherian, then $T$ is f-noetherian.
\end{lemma}

\begin{proof}
Let $I=S\!\setminus\!T.$
Since $S$ is f-noetherian, we have that the Rees quotient $S/I=T\cup\{0\}$ is f-noetherian, and hence $T$ is f-noetherian by Corollary \ref{adjoin zero}.
\end{proof}

Ideals of f-noetherian semigroups are not in general f-noetherian; see Example \ref{com reg ex}, for instance.
Monoid ideals, however, do inherit the property of being f-noetherian.

\begin{lemma}
\label{monoid ideal}
Let $S$ be a semigroup with an ideal $I$ that has an identity $e.$  If $S$ is f-noetherian, then so is $I.$
\end{lemma}

\begin{proof}
Let $\rho$ be a right congruence of finite index on $I.$  We define a right congruence $\rho^{\prime}$ on $S$ by $$s\,\rho^{\prime}\,t\iff es\,\rho\,et.$$
Since $\rho$ partitions $I$ into finitely many classes and $es\in I$ for all $s\in S,$ therefore $\rho^{\prime}$ has finite index.
Since $S$ is f-noetherian, $\rho^{\prime}$ is generated by a finite set $X.$  We claim that $\rho$ is generated by the finite set $$Y=\{(ex, ey) : (x, y)\in X\}.$$
Indeed, let $a\,\rho\,b$ with $a\neq b.$  We have that $ea=a\,\rho\,b=eb,$ so $a\,\rho^{\prime}\,b.$ 
Therefore, there exists an $X$-sequence $$a=x_1s_1, y_1s_1=x_2s_2, \dots, y_ks_k=b.$$
We now have a $Y$-sequence $$a=(ex_1)(s_1e), (ey_1)(s_1e)=(ex_2)(s_2e), \dots, (ey_k)(s_ke)=b,$$
so $(a, b)$ is a consequence of $Y.$
\end{proof}

\begin{remark}
\label{monoid ideal remark}
Analogues of Lemma \ref{monoid ideal} hold for the properties of being finitely connected and being right noetherian, and the proofs are essentially the same.
\end{remark}

It is well known that maximal subgroups of semigroups coincide with the group $\mathcal{H}$-classes. 
Sch{\"u}tzenberger showed in \cite{Schutzenberger} how one can assign a group to an arbitrary $\mathcal{H}$-class,
and we now briefly describe his construction; one may consult \cite{Lallement} for more details and basic properties of Sch{\"u}tzenberger groups.\par
Let $S$ be a semigroup and let $H$ be an $\mathcal{H}$-class of $S.$
The {\em right stabiliser} of $H$ is the set $$\text{Stab}(H)=\{s\in S^1 : Hs=H\}.$$
Clearly $\text{Stab}(H)$ is a submonoid of $S^1.$
Define a relation $\sigma(H)$ on $\text{Stab}(H)$ by $$(s, t)\in\sigma(H)\iff hs=ht\text{ for all }h\in H.$$
It is easy to see that $\sigma(H)$ is a congruence on $\text{Stab}(H),$ 
and it turns out that the quotient $\Gamma(H)=\text{Stab}(H)/\sigma(H)$ is a group.
We call $\Gamma(H)$ the {\em Sch{\"u}tzenberger group} of $H$.
Note that $|\Gamma(H)|=|H|,$ and if $H$ is a group $\mathcal{H}$-class, then it is isomorphic to $\Gamma(H).$\par
The following result states that within an $\mathcal{R}$-class with only finitely many $\mathcal{H}$-classes, the Sch{\"u}tzenberger group of each of those $\mathcal{H}$-classes is finitely generated.

\begin{prop}
\label{Schutzenberger}
Let $S$ be a semigroup with an $\mathcal{R}$-class $R$ such that $R$ is a finite union of $\mathcal{H}$-classes, and let $H$ be an $\mathcal{H}$-class of $R.$
If $S$ is f-noetherian, then the Sch{\"u}tzenberger group $\Gamma(H)$ is finitely generated.
\end{prop}

\begin{proof}
Let $\rho$ be the right congruence of finite index on $S^1$ given by 
$$s\,\rho\,t\iff 
\begin{cases} 
   Hs=Ht\subseteq R,\\
   Hs\cup Ht\subseteq S\!\setminus\!R.
\end{cases}$$
Now $S^1$ is f-noetherian by Corollary \ref{adjoin identity}, so $\rho$ is generated by a finite set $X.$
Fix $h\in H.$  For each $(x, y)\in\overline{X}$ such that $Hx\subseteq R$ (so $Hx=Hy$), 
choose $\alpha(x, y)\in\text{Stab}(H)$ such that $hx=h\alpha(x, y)y,$
and let $$A=\{[\alpha(x, y)]_{\sigma} : (x, y)\in\overline{X}, Hx\subseteq R\},$$
where $\sigma=\sigma(H).$  We claim that $\Gamma(H)$ is generated by the finite set $A.$\par 
Indeed, let $s\in\text{Stab}(H).$  Then $s\,\rho\,1,$ so there exists an $X$-sequence
$$s=x_1s_1, y_1s_1=x_2s_2, \dots y_ks_k=1.$$
Since $H=Hs=(Hx_1)s_1,$ we have that $Hx_1\subseteq R.$
Therefore, we have that 
$$H=(Hx_1)s_1=(Hy_1)s_1=(Hx_2)s_2,$$ 
so $Hx_2\subseteq R.$  Continuing in this way, we have that $Hx_i\subseteq R$ for all $i\in\{1, \dots, k\}.$
Let $\alpha_i=\alpha(x_i, y_i),$ and choose $\beta_i\in S$ such that $h\alpha_i=\beta_ih.$
We now have that $$hs=hx_1s_1=h\alpha_1y_1s_1=\beta_1hy_1s_1=\dots=\beta_1\dots\beta_khy_ks_k=\beta_1\dots\beta_kh=h\alpha_1\dots\alpha_k.$$
Therefore, we have that $s\,\sigma\,\alpha_1\dots\alpha_k,$ and hence 
$$[s]_{\sigma}=[\alpha_1\dots\alpha_k]_{\sigma}=[\alpha_1]_{\sigma}\dots[\alpha_k]_{\sigma}\in\langle A\rangle,$$
as required.
\end{proof}

\begin{corollary}
\label{maximal subgroup}
Let $S$ be a semigroup with a maximal subgroup $G$ whose $\mathcal{R}$-class is a finite union of $\mathcal{H}$-classes.  If $S$ is f-noetherian, then $G$ is finitely generated.
\end{corollary}

\begin{remark}
Maximal subgroups of f-noetherian semigroups are not in general finitely generated.
For example, the full transformation semigroup $\mathcal{T}_X$ on an infinite set $X$ is f-noetherian by Proposition \ref{cyclic diagonal act}, 
but its maximal subgroup $S_X$ is not finitely generated.
\end{remark}
 
It is an open problem whether the direct product of two right noetherian monoids is right noetherian \cite[Open Problem 4.8]{Miller3}.
On the other hand, it was shown in \cite[Proposition 4.2]{Dandan} that the property of being finitely connected {\em is} preserved under direct products of monoids.
We now state and prove the analogue of this result for the property of being f-noetherian.

\begin{thm}
\label{dp}
Let $M$ and $N$ be two monoids.  Then $M\times N$ is f-noetherian if and only if both $M$ and $N$ are f-noetherian.
\end{thm}

\begin{proof}
The direct implication follows from Lemma \ref{quotient}, so we just need to prove the converse.\par
Let $\rho$ be a right congruence of finite index on $M\times N,$ and let $C_i, i\in I,$ be its classes.
For each $m\in M,$ we define a right congruence $\rho_m^N$ on $N$ by
$$a\,\rho_m^N\,b\iff(m, a)\,\rho\,(m, b).$$
In a similar way we define a right congruence $\rho_n^M$ on $M$ for each $n\in N.$
Clearly $\rho_m^N$ and $\rho_n^M$ have finite index for every $m\in M$ and $n\in N.$
Let $D_j, j\in J,$ be the $\rho_{1_M}^N$-classes.
For each $j\in J,$ choose $d_j\in D_j.$\par
Let $Q$ denote the set of pairs $(i, j)\in I\times J$ such that $(m, d_j)\in C_i$ for some $m\in M.$
For each pair $(i, j)\in Q,$ choose $\alpha_{i, j}\in M$ such that $(\alpha_{i, j}, d_j)\in C_i.$
We now define a set 
$$H=\bigl\{\bigl((\alpha_{i, j}, d_j), (\alpha_{i, k}, d_k)\bigr) : (i, j), (i, k)\in Q, j\neq k\bigr\}.$$
Since $M$ and $N$ are f-noetherian, we have that $\rho_{1_M}^N$ is generated by some finite set $X$,
and each $\rho_j^M=\rho_{d_j}^M$ is generated by some finite set $X_j.$
We define a set 
$$Y=\bigl\{\bigl((1_M, x), (1_M, y)\bigr) : (x, y)\in X\bigr\},$$
and, for each $j\in J,$ we let 
$$Y_j=\bigl\{\bigl((x, d_j), (y, d_j)\bigr) : (x, y)\in X_j\bigr\}.$$
We claim that $\rho$ is generated by the finite set 
$$Z=H\cup Y\cup\biggl(\bigcup_{j\in J}Y_j\biggr).$$
Indeed, let $(m, u), (n, v)\in C_i, i\in I,$ with $(m, u)\neq(n, v).$\par
Now, there exist $j, k\in J$ such that $u\in D_j$ and $v\in D_k.$
Therefore, there exists an $X$-sequence connecting $u$ and $d_j,$ 
and hence there clearly exists a $Y$-sequence connecting $(1_M, u)$ and $(1_M, d_j).$
Multiplying every term of this latter sequence on the right by $(m, 1_N),$ 
we have a $Y$-sequence connecting $(m, u)$ and $(m, d_j).$
Since $(m, d_j)\in C_i,$ we have that $(m, d_j)\,\rho\,(\alpha_{i, j}, d_j),$ so $m\,\rho_j\,\alpha_{i, j}.$
Therefore, by a similar argument as above, there exists a $Y_j$-sequence connecting $(m, d_j)$ and $(\alpha_{i, j}, d_j).$
Hence, we have that $\bigl((m, u), (\alpha_{i, j}, d_j)\bigr)$ is a consequence of $Y\cup Y_j.$
A similar argument proves that $\bigl((n, v), (\alpha_{i, k}, d_k)\bigr)$ is a consequence of $Y\cup Y_k.$
Now, either $j=k$ or $\bigl((\alpha_{i, j}, d_j), (\alpha_{i, k}, d_k)\bigr)\in H,$
so it follows that $\bigl((m, u), (n, v)\bigr)$ is a consequence of $Z.$
\end{proof}

\begin{remark}
The direct product of two f-noetherian semigroups is not necessarily f-noetherian.  
Indeed, it was shown in \cite[Example 4.3]{Dandan} that $\mathbb{N}\times\mathbb{N}$ is not finitely connected.
\end{remark}

\section{\large{Transformation semigroups}\nopunct}

In this section we prove that certain infinite transformation semigroups are f-noetherian.
The notion of the {\em diagonal act} of a semigroup will turn out to be useful.
For any semigroup $S$ and $n\in\mathbb{N},$ the Cartesian product of $n$ copies of $S$ can be made into an $S$-act by defining
$$(x_1, \dots, x_n)s=(x_1s, \dots, x_ns)$$
for all $x_1, \dots, x_n, s\in S.$
If $n=2,$ we call it the {\em diagonal $S$-act}, and for $n\geq 3$ we refer to it as the {\em $n$-diagonal $S$-act} and denote it by $S^{(n)}.$\par
It is easy to see that any set that generates the diagonal $S$-act is also a generating set for the universal right congruence on $S.$
Gallagher showed in \cite{Gallagher1} that infinite semigroups from various `standard' semigroup classes, 
such as commutative, inverse, idempotent, cancellative, completely regular and completely simple, do {\em not} have finitely generated diagonal acts.
However, it was proved in \cite{Gallagher2} that the diagonal act is cyclic for various transformation semigroups on an infinite set.

\begin{thm}\cite[Table 1]{Gallagher2}
\label{transformations}
Let $X$ be an infinite set, and let $S$ be any of the following transformation semigroups on $X$:
\begin{itemize}
\item $\mathcal{B}_X$ (the monoid of binary relations); 
\item $\mathcal{T}_X$ (the full transformation monoid);
\item $\mathcal{P}_X$ (the monoid of partial transformations); 
\item $\mathcal{F}_X$ (the monoid of full finite-to-one transformations).
\end{itemize}
Then the diagonal $S$-act is cyclic.
\end{thm}

We show that a semigroup $S$ whose diagonal act is cyclic is f-noetherian.  In order to prove this, we require the following result of Gallagher.

\begin{prop}\cite[Propositions 3.1.11 and 3.1.12]{Gallagher}
\label{diagonal act}
Let $S$ be a non-trivial semigroup with a cyclic diagonal act, so that $S\times S=(a, b)S$ for some $a, b\in S.$
\begin{enumerate}
 \item The subsemigroup generated by $a$ and $b$ is free of rank two.
 \item For each $n\in\mathbb{N},$ if we list the elements of $\{a, b\}^n$ (the set of words of length $n$ over $\{a, b\}$) as $w_1, \dots, w_{2^n},$ then $S^{(2^n)}=(w_1, \dots, w_{2^n})S.$
\end{enumerate}
\end{prop}

\begin{prop}
\label{cyclic diagonal act}
Let $S$ be a semigroup.  If the diagonal $S$-act is cyclic, then $S$ is f-noetherian.
\end{prop}

\begin{proof}
If $S$ is trivial, then it is obviously f-noetherian, so assume that $S$ is non-trivial.  
Since the diagonal $S$-act is cyclic, there exist $a, b\in S$ such that $S\times S=(a, b)S,$
and hence $(a, b)$ generates the universal right congruence on $S.$\par
We claim that the universal right congruence is the only right congruence of finite index on $S.$
So, let $\rho$ be a right congruence of index $n$ on $S.$ 
List the elements of $\{a, b\}^n,$ which are distinct by Proposition \ref{diagonal act}(1), as $w_1, \dots, w_{2^n}.$
By the Pigeonhole Principle, there exist $i, j\in\{1, \dots, 2^n\}$ with $i\neq j$ such that $w_i\,\rho\,w_j.$
We have that $S^{(2^n)}=(w_1, \dots, w_{2^n})S$ by Proposition \ref{diagonal act}(2).
It follows that $S\times S=(w_i, w_j)S,$ and hence $\rho$ is the universal right congruence.
\end{proof}

The symmetric inverse monoid $\mathcal{I}_X$ on an infinite set $X$ does not have a finitely generated diagonal act \cite[Theorem 2.5]{Gallagher2}.  However, it is f-noetherian:

\begin{prop}
\label{symmetric inverse}
Let $X$ be an infinite set.  Then the symmetric inverse monoid $\mathcal{I}_X$ is f-noetherian.
\end{prop}

\begin{proof}
Letting $1$ and $0$ denote the identity and zero of $I_X$ respectively, we have that the universal right congruence on $\mathcal{I}_X$ is generated by the pair $(1, 0).$
We claim that this is the only right congruence of finite index on $\mathcal{I}_X.$
So, let $\rho$ be a right congruence of index $n$ on $\mathcal{I}_X.$
Partition $X$ as $X=\bigcup_{i=1}^{n+1}X_i,$ where the $X_i$ are disjoint subsets of $X,$ each with the same cardinality as $X,$ and define bijections $\alpha_i : X\to X_i.$  
By the Pigeonhole Principle, there exist $i, j\in\{1, \dots, n+1\}$ with $i\neq j$ such that $\alpha_i\,\rho\,\alpha_j.$
But then we have that $$1=\alpha_i\alpha_i^{-1}\,\rho\,\alpha_j\alpha_i^{-1}=0,$$ so $\rho=\langle(1, 0)\rangle$ is the universal right congruence.
\end{proof}

\begin{remark}
Let $X$ be an infinite set.  The semilattice of idempotents of $\mathcal{I}_X$ is isomorphic to the semigroup $P(X)$ of all subsets of $X$ under intersection.
It is finitely connected since it is a monoid with a zero.  However, it is not f-noetherian.
Indeed, the subsemigroup $P(X)\!\setminus\!\{X\}$ has the infinite set of maximal elements $\{X\!\setminus\!\{x\} : x\in X\},$
so it is not finitely connected by \cite[Corollary 5.7]{Dandan}, and hence not f-noetherian.
It follows from Theorem \ref{large} that $P(X)$ is not f-noetherian.
\end{remark}

\begin{remark}
The symmetric group $\mathcal{S}_X$ on an infinite set $X$ is not f-noetherian by Lemma \ref{group}, since it is not finitely generated.
\end{remark}

For a set $X,$ we denote the semigroups of full surjective transformations and full injective transformations on $X$ by ${Surj}_X$ and ${Inj}_X$ respectively.

\begin{prop}
Let $X$ be an infinite set.  Then the semigroups ${Surj}_X$ and ${Inj}_X$ are not f-noetherian.
\end{prop}

\begin{proof}
We note that $S_X$ is a subsemigroup of both ${Surj}_X$ and ${Inj}_X.$
It was shown in the proofs of \cite[Theorems 4.4.2 and 4.4.4]{Gallagher} that ${Surj}_X\!\setminus\!S_X$ is an ideal of ${Surj}_X$ and ${Inj}_X\!\setminus\!S_X$ is an ideal of ${Inj}_X.$
Therefore, since $S_X$ is not f-noetherian, we have that ${Surj}_X$ and ${Inj}_X$ are not f-noetherian by Lemma \ref{complement ideal}.
\end{proof}

\section{\large{Finite generation}\nopunct}

We now turn to the main theme of the paper, which is to investigate the relationship between the properties of being f-noetherian and being finitely generated.
We have already seen in the previous section that there exist f-noetherian semigroups that are non-finitely generated (indeed, uncountable).
In this section we show that the class of finitely generated semigroups is strictly contained within the class of f-noetherian semigroups.
We also prove that the properties coincide for semigroups with finitely many $\mathcal{L}$-classes.

\begin{prop}
\label{fg}
Let $S$ be a semigroup.  If $S$ is finitely generated, then it is f-noetherian.
\end{prop}

\begin{proof}
Let $X$ be a finite generating set for $S,$ 
and let $\rho$ be a right congruence of finite index on $S$ with classes $C_i, i\in I.$
For each $i\in I,$ fix $\alpha_i\in C_i.$  
We claim that $\rho$ is generated by the finite set
$$H=\{(x, \alpha_i) : i\in I, x\in X\cap C_i\}\cup\{(\alpha_ix, \alpha_j) : i, j\in I, C_ix\subseteq C_j\}.$$
Indeed, let $s\,\rho\,t$ with $s\neq t.$
We have that $s=x_1\dots x_n$ for some $x_i\in X.$  
Let $x_1\in C_{i_1}$ and, for each $j\in\{1, \dots, n-1\},$ let $C_{i_j}x_{j+1}\subseteq C_{i_{j+1}}$ and $u_j=x_{j+1}\dots x_n.$ 
We then have the following $H$-sequence connecting $s$ and $\alpha_{i_n}$:
$$s=x_1u_1, \alpha_{i_1}u_1=(\alpha_{i_1}x_2)u_2, \alpha_{i_2}u_2=(\alpha_{i_2}x_3)u_3, \dots, \alpha_{i_{n-1}}u_{n-1}=\alpha_{i_{n-1}}x_n, \alpha_{i_n}.$$ 
Similarly, there exists $i\in I$ such that $(t, \alpha_{i})$ is a consequence of $H.$
Since $s\,\rho\,t,$ we have that $i_n=i,$ so $(s, t)$ is a consequence of $H.$
\end{proof}

In the case that a semigroup has finitely many $\mathcal{L}$-classes, the converse of Proposition \ref{fg} holds.

\begin{prop}
\label{L-classes}
Let $S$ be a semigroup with finitely many $\mathcal{L}$-classes.  If $S$ is f-noetherian, then it is finitely generated.
\end{prop}

\begin{proof}
Recall that $\mathcal{L}$ is a right congruence on $S.$
Since $S$ is f-noetherian, we have that $\mathcal{L}=\langle X\rangle$ for some finite set $X\subseteq\mathcal{L}.$
For each $(x, y)\in\overline{X},$ choose $\alpha(x, y)\in S$ such that $x=\alpha(x, y)y.$
Let $L_i, i\in I,$ be the $\mathcal{L}$-classes of $S,$ and for each $i\in I$ fix $b_i\in L_i.$
We claim that $S$ is generated by the finite set 
$$A=\{\alpha(x, y) : (x, y)\in\overline{X}\}\cup\{b_i : i\in I\}.$$
Indeed, if $s\in S,$ then $s\in L_i$ for some $i\in I.$ 
Let $b=b_i.$
If $s=b,$ then $s\in A,$ so assume that $s\neq b.$
Since $s\,\mathcal{L}\,b,$ there exists an $X$-sequence
$$s=x_1s_1, y_1s_1=x_2s_2, \dots, y_ks_k=b,$$
where $(x_i, y_i)\in\overline{X}$ and $s_i\in S^1$ for $1\leq i\leq k.$
Letting $\alpha_i=\alpha(x_i, y_i)$ for $1\leq i\leq k,$ we have that
$$s=(\alpha_1y_1)s_1=\alpha_1(y_1s_1)=\alpha_1(x_2s_2)=\cdots=\alpha_1\dots\alpha_k(y_ks_k)=\alpha_1\dots\alpha_kb\in\langle A\rangle,$$
as required.
\end{proof}

\begin{remark}
\label{left simple}
In a left simple semigroup $S,$ Green's $\mathcal{L}$-relation coincides with the universal relation.
Therefore, if $S$ is finitely connected, then the same argument as the one in the proof of Proposition \ref{L-classes} shows that $S$ is finitely generated.
\end{remark}

Being f-noetherian does not imply finite generation under the weaker condition that a semigroup has finitely many $\mathcal{D}$-classes.
We demostrate this by providing an example of a $0$-bisimple inverse monoid that is f-noetherian but not finitely generated.

\begin{ex}
\label{polycyclic}
Let $X$ be an infinite set, and let $M$ be the monoid with zero defined by the presentation 
$$\langle X, X^{-1}\,|\,xx^{-1}=1, xy^{-1}=0\,(x, y\in X, x\neq y)\rangle.$$
The monoid $M$ is called the {\em polycyclic monoid} on $X,$ and is easily seen to be $0$-bisimple and inverse.
The universal right congruence on $M$ is generated by the pair $(1, 0).$
We claim that this is the only right congruence of finite index on $M.$
Indeed, if $\rho$ is a right congruence of finite index on $M,$ there exist distinct $x, y\in X$ with $x\,\rho\,y.$
But then we have that $$1=xx^{-1}\,\rho\,yx^{-1}=0,$$ so $\rho=\langle(1, 0)\rangle$ is the universal congruence.
\end{ex}

\section{\large{Completely regular semigroups}\nopunct}

In this section we consider completely regular semigroups.
A {\em completely regular semigroup} is a semigroup which is a union of groups.\par
We shall make use of the following construction.
Let $Y$ be a semilattice and let $(S_{\alpha})_{\alpha\in Y}$ be a family of disjoint semigroups, indexed by $Y,$
such that $S=\bigcup_{\alpha\in Y}S_{\alpha}$ is a semigroup.
If $S_{\alpha}S_{\beta}\subseteq S_{\alpha\beta}$ for all $\alpha, \beta\in Y,$
then $S$ is called a {\em semilattice of semigroups}, and we denote it by $S=\mathcal{S}(Y, S_{\alpha}).$\par
We have the following structure theorem for completely regular semigroups.

\begin{thm}\cite[Theorem 4.1.3]{Howie}
\label{com reg structure}
Every completely regular semigroup is a semilattice of completely simple semigroups.
\end{thm}

The following necessary and sufficient conditions for a semilattice to be finitely connected were provided in \cite{Dandan}.

\begin{prop}\cite[Corollary 5.7]{Dandan}
\label{fc semilattice}
A semilattice $Y$ is finitely connected if and only if there exists a finite set $X\subseteq Y$ such that $Y=XY^1$ and $Y$ has a zero element.
\end{prop}

We now prove that f-noetherian semilattices are finite.

\begin{prop}
\label{semilattice}
Let $Y$ be a semilattice.  If $Y$ is f-noetherian, then it is finite.
\end{prop}

\begin{proof}
Suppose for a contradiction that $Y$ is infinite, and let $Y=Y_1.$  
By Proposition \ref{fc semilattice}, there exists a finite set $X_1\subseteq Y_1$ such that $Y_1=X_1Y_1^1.$
Choose $x_{11}\in X_1$ such that $x_{11}Y^1$ is infinite.
Let $$Z_{11}=\{z\in X_1 : (x_{11}z)Y^1\text{ is finite}\}.$$
If $Z_{11}\cup\{x_{11}\}\neq X_1,$ choose $x_{12}\in X_1\!\setminus\!(Z_{11}\cup\{x_{11}\}).$
Now let $$Z_{12}=\{z\in X_1 : (x_{11}x_{12}z)Y^1\text{ is finite}\}.$$
Notice that $Z_{11}\subset Z_{12}.$
Continue this procedure until we obtain sets $Z_1=Z_{1n_1}$ and $X_1^{\prime}=\{x_{11}, \dots, x_{1n_1}\}\subseteq X_1$ such that $X_1=Z_1\cup X_1^{\prime}.$
Let $e_1=x_{11}\dots x_{1n_1}.$  Note that $(e_1z)Y^1$ is finite for every $z\in Z_1.$\par 
Now consider the semilattice $Y_2=Y_1\!\setminus\!X_1.$  
We have that $Y_2$ is f-noetherian by Proposition \ref{large}, so there exists a finite set $X_2\subseteq Y_2$ such that $Y_2=X_2Y_2^1$ by Proposition \ref{fc semilattice}.
There exists $x\in X_2$ such that $(xe_1)Y^1$ is infinite; choose $x_{21}\in X_2$ with this property,
and let $$Z_{21}=\{z\in X_2 : (e_1x_{21}z)Y^1\text{ is finite}\}.$$
In a similar way as above we construct sets $X_2^{\prime}=\{x_{21}, \dots, x_{2n_2}\}$ and $Z_2=X_2\!\setminus\!X_2^{\prime}$ such that,
letting $e_2=e_1x_{21}\dots x_{2n_2},$ we have $e_2Y^1$ is infinite and $(e_2z)Y^1$ is finite for all $z\in Z_1\cup Z_2.$\par 
Continuing this process ad infinitum, we obtain semilattices $$Y_1\supset Y_2\supset\dots,$$
and, for each $k\in\mathbb{N},$ we have the following:
\begin{itemize}
 \item a finite set $X_k\subseteq Y_k$ such that $Y_k=X_kY_k^1$;
 \item sets $X_k^{\prime}=\{x_{k_1}\dots x_{kn_k}\}$ and $Z_k=X_k\!\setminus\!x_k^{\prime}$;
 \item an element $e_k=e_1\dots e_{k-1}x_{k1}\dots x_{kn_k}$ such that $e_kY^1$ is infinite and $(e_kz)Y^1$ is finite for every $z\in\bigcup_{i=1}^kZ_k.$
\end{itemize}
Now let $U$ be the subsemilattice of $Y$ generated by the set $\bigcup_{i=1}^{\infty}X_i^{\prime}.$
We claim that the complement $Y\!\setminus\!U$ is the ideal 
$$I=\{g\in Y : g\ngeq u\text{ for any }u\in U\}.$$
Clearly $I\subseteq Y\!\setminus\!U.$  
Now let $e\in Y\!\setminus\!U,$ and suppose that $e\not\in I,$ so $e\geq u$ for some $u\in U.$
We then have that $e\in Z_j$ and $u\in\langle\bigcup_{i=1}^kX_i^{\prime}\rangle$ for some $j, k\in\mathbb{N}.$
Letting $n=\text{max}(j, k),$ we have $e\geq ue_je_k=e_n.$
But then we have that $e_nY^1=(e_ne)Y^1$ is finite, which is a contradiction, so $e\in I$ and $Y\!\setminus\!U=I.$\par
It now follows from Lemma \ref{complement ideal} that $U$ is f-noetherian.
However, $U$ does not contain a zero, contradicting Proposition \ref{fc semilattice}, and hence $Y$ is finite.
\end{proof}

Since the semilattice $Y$ is a homomorphic image of $S=\mathcal{S}(Y, S_{\alpha}),$ Lemma \ref{quotient} and Proposition \ref{semilattice} together yield the following corollary.

\begin{corollary}
\label{finite semilattice}
Let $S=\mathcal{S}(Y, S_{\alpha})$ be a semilattice of semigroups.  If $S$ is right noetherian, then $Y$ is finite.
\end{corollary}

The following example shows that for a completely regular semigroup to be f-noetherian,
it is not required that every completely simple semigroup in its semilattice decomposition be f-noetherian.

\begin{ex}
\label{com reg ex}
Let $S$ be the semigroup defined by the presentation $$\langle a, b, c\,|\,ab=ba=1, ac=c^2=c\rangle.$$
We have that $S$ is a semilattice of two semigroups $S_{\alpha}=\langle a, b\rangle\cong\mathbb{Z}$ and $S_{\beta}=\{ca^i, cb^i : i\in\mathbb{N}\}$ where $\alpha>\beta.$
Now $S$ is f-noetherian since it is finitely generated.  However, $S_{\beta}$ is an infinite right zero semigroup, and hence not finitely connected by \cite[Corollary 3.6]{Dandan}.
\end{ex}

Green's relation $\mathcal{H}$ is not in general a congruence on a completely regular semigroup.
A completely regular semigroup for which $\mathcal{H}$ is a congruence is called a {\em cryptogroup}.
Examples of cryptogroups include completely simple semigroups, Clifford semigroups and bands.
We shall show that every f-noetherian cryptogroup is finitely generated.
We first prove the following technical lemma.

\begin{lemma}
\label{congruence contained in L}
Let $S=\mathcal{S}(Y, S_{\alpha})$ be a semilattice of semigroups, and suppose that $S$ has a congruence $\rho\subseteq\mathcal{L}$ with the following property:
for each $\alpha\in Y,$ if the semigroup $S_{\alpha}/\rho_{\alpha}$ (where $\rho_{\alpha}=\rho|_{S_{\alpha}}$) is f-noetherian, then it is finite.
If $S$ is f-noetherian, then it is finitely generated.
\end{lemma}

\begin{proof}
The semilattice $Y$ is finite by Corollary \ref{finite semilattice}.
Let $T=S/\rho.$  Then $T$ is a semilattice $\mathcal{S}(Y, T_{\alpha}),$ where $T_{\alpha}=S_{\alpha}/\rho_{\alpha},$ and $T$ is f-noetherian by Lemma \ref{quotient}. 
We prove that $T$ is finite by induction on the order of $Y.$  
If $|Y|=1,$ then $T$ is finite by assumption.
Suppose that $|Y|>1,$ and choose a maximal element $\beta\in Y.$
We have that $T_{\beta}$ is f-noetherian by Lemma \ref{complement ideal}, and hence finite by assumption.
Therefore, by Theorem \ref{large}, we have that $T\!\setminus\!T_{\beta}=\mathcal{S}(Y\!\setminus\!\{\beta\}, T_{\alpha})$ is f-noetherian. 
By induction we have that $T\!\setminus\!T_{\beta}$ is finite, and hence $T$ is finite.\par
It now follows that $S$ has finitely many $\mathcal{L}$-classes, and is hence finitely generated by Proposition \ref{L-classes}.
\end{proof}

\begin{thm}
\label{cryptogroup}
Let $S$ be a cryptogroup.  If $S$ is f-noetherian, then it is finitely generated.
\end{thm}

\begin{proof}
We have that $S$ is a semilattice $\mathcal{S}(Y, S_{\alpha})$ of completely simple semigroups by Theorem \ref{com reg structure}.
For each $\alpha\in Y,$ the quotient $S_{\alpha}/\mathcal{H}_{\alpha}$ is a rectangular band.
By \cite[Corollary 6.4]{Dandan}, any finitely connected rectangular band is finite.
Hence, we have that $S$ is finitely generated by Lemma \ref{congruence contained in L}.
\end{proof}

\begin{corollary}
\label{band}
A band $B$ is f-noetherian if and only if it is finite.
\end{corollary}

\section{\large{Commutative semigroups}\nopunct}

For commutative semigroups, being (right) noetherian is equivalent to being finitely generated.
Finitely connected commutative semigroups, however, are not necessarily finitely generated.
Indeed, in Section 5 we saw that there exist infinite finitely connected semilattices.
In this section we investigate how the properties of being f-noetherian and being finitely generated relate to one another for commutative semigroups.
We begin by presenting an example of a cancellative idempotent-free commutative semigroup that is f-noetherian but not finitely generated. 

\begin{ex}
\label{commutative ex}
Let $A=\langle a\rangle\cong\mathbb{N},$ and let $B=\mathbb{N}_0\times\mathbb{Z}$ be the semigroup with multiplication given by $$(m, i)(n, j)=(m+n+1, i+j).$$
Let $S=A\cup B,$ and define a multiplication on $S,$ extending those on $A$ and $B,$ as follows:
$$a^j(m, i)=(m, i)a^j=(m, i-j).$$
Letting $b_i$ represent the generator $(0, i), i\in\mathbb{Z},$ the semigroup $S$ is defined by the presentation 
$$\langle a, b_i\,(i\in\mathbb{Z})\,|\,ab_i=b_ia=b_{i-1}, b_ib_j=b_0b_{i+j}\,(i, j\in\mathbb{Z})\rangle.$$
It is easy to see that $S$ is commutative, cancellative, idempotent-free and not finitely generated.
We now show that $S$ is f-noetherian.\par 
Let $\rho$ be a congruence of finite index on $S.$  
We have that $\rho|_A$ is generated by some pair $(a^s, a^t)$ with $s<t$; let $r=t-s.$
Now choose $m_0, n_0\in\mathbb{N}_0$ with $n_0>m_0$ such that $(n_0, 0)\,\rho\,(m_0, z)$ for some $z\in\mathbb{Z}.$
We let $$X=\{(a^s, a^t)\}\cup\{\bigl((n_0, 0), (m_0, z)\bigr)\},$$
and let $U$ denote the finite set
$$\{a^i : 1\leq i<t\}\cup\{(k, p) : 0\leq k<n_0, 0\leq q<r\}.$$
We show that $\rho$ is generated by the set $Y=X\cup\bigl(\rho\cap(U\times U)\bigr).$
We first make the following claim.
\begin{claim*}
\begin{enumerate}
 \item For each $n\geq n_0$ and $i\in\mathbb{Z},$ we have $\bigl((n, i), (n-(n_0-m_0), i+z)\bigr)$ is a consequence of $\bigl((n_0, 0), (m_0, z)\bigr).$
 \item For each $n\in\mathbb{N}_0$ and $i\in\mathbb{Z},$ there exists some $0\leq q<r$ such that $\bigl((n, i), (n, q)\bigr)$ is a consequence of $(a^s, a^t).$
\end{enumerate}
\end{claim*}
\begin{proof}
(1) We have the sequence $$(n, i)=(n_0, 0)(n-n_0-1, i),\,(m_0, z)(n-n_0-1, i)=(n-(n_0-m_0), i+z).$$
(2) We have that $i=pr+q$ for some $p\in\mathbb{Z}$ and $0\leq q<r.$
Now the pair $(a^s, a^{s+|p|r})$ is a consequence of $(a^s, a^t)$.
Therefore, if $p\geq 0,$ we obtain $(n, q)=a^{s+pr}(n, i+s)$ from $(n, i)=a^s(n, i+s),$ 
and if $p<0,$ we obtain $(n, q)=a^s(n, i+(s-pr))$ from $(n, i)=a^{s-pr}(n, i+(s-pr)).$
\end{proof}
Now let $s\,\rho\,t$ with $s\neq t.$
If $s\in A,$ then there exists $i\in\{s, \dots, t-1\}$ such that $(s, a^i)$ is a consequence of $(a^s, a^t).$
If $s=(n, i)\in B,$ then it follows from (1) of the above claim that there exist some $m_0\leq k<n_0$ and $p\in\mathbb{Z}$ 
such that $\bigl((n, i), (k, p)\bigr)$ is a consequence of $\bigl((n_0, 0), (m_0, z)\bigr),$
and by (2) there exists some $0\leq q<r$ such that $\bigl((k, i), (k, q)\bigr)$ is a consequence of $(a^s, a^t).$
Therefore, in either case, there exists some $u\in U$ such that $(s, u)$ is a consequence of $X.$
Similarly, there exists $v\in U$ such that $(t, v)$ is a consequence of $X.$
Since $(u, v)\in\rho\cap(U\times U),$ we have that $(s, t)$ is a consequence of $Y,$ as required.
\end{ex}

\begin{defn}
An {\em archimedean semigroup} is a commutative semigroup $S$ in which, for each $a, b\in S,$ there exists $n>0$ such that $a^n\leq b.$
\end{defn}

We have the following structure theorem for commutative semigroups.

\begin{thm}\cite[Theorem IV.2.2]{Grillet}
\label{archimedean decomposition}
Every commutative semigroup is a semilattice of archimedean semigroups.
\end{thm}

Below we provide a characterisation of archimedean semigroups that have an idempotent.
In general, archimedean semigroups can have a rather complex structure.
We refer the reader to \cite[Chapter IV]{Grillet} for more information.

\begin{lemma}\cite[Proposition IV.2.3]{Grillet}
\label{Grillet, archimedean with idempotent}
A commutative semigroup $S$ is archimedean with idempotent if and only if $S$ is either a group or an ideal extension of a group by a nilpotent semigroup.
\end{lemma}

In the following we show that finitely connected archimedean semigroups are finitely generated.
We begin with the following lemma.

\begin{lemma}
\label{nilpotent}
Let $T$ be a nilpotent semigroup.  If $T$ is finitely connected, then it is finite.
\end{lemma}

\begin{proof}
Since $T$ is finitely connected, there exists a finite set $X\subseteq T$ such that $T=XT^1.$
Let $U=\langle X\rangle.$ 
For each $x\in X,$ let $$m(x)=\text{min}\{n\in\mathbb{N} : x^n=0\},$$ and let $N=\prod_{x\in X}m(x).$  It can be easily shown that $|U|\leq N.$
We claim that $T=U.$  Suppose for a contradiction that $T\neq U,$ and let $a\in T\!\setminus\!U.$
We have that $a=x_1t_1$ for some $x_1\in X$ and $t_1\in T^1.$  Since $a\notin U,$ we have that $t_1\in T\!\setminus\!U.$  
By a similar argument, there exist $x_2\in X$ and $t_2\in T\!\setminus\!U$ such that $t_1=x_2t_2.$
Continuing in this way, for each $n\in\mathbb{N}$ there exist $x_1, \dots, x_n\in X$ and $t_n\in T\!\setminus\!U$ such that $a=(x_1\dots x_n)t_n.$
However, we have that $x_1\dots x_N=0$ and hence $a=0,$ which is a contradiction. 
\end{proof}

\begin{prop}
\label{archimedean}
Let $S$ be an archimedean semigroup.  If $S$ is finitely connected, then it is finitely generated.
\end{prop}

\begin{proof}
Suppose first that $S$ has an idempotent.
If $S$ is a group, then it is finitely generated by Lemma \ref{group}, so assume that $S$ is not a group. 
By Lemma \ref{Grillet, archimedean with idempotent}, there exists a group $G$ that is an ideal of $S$ such that $T=S/G$ is a nilpotent semigroup.
We have that $G$ is finitely generated by Remark \ref{monoid ideal remark} and Lemma \ref{group}, and $T$ is finite by Remark \ref{quotient remark} and Lemma \ref{nilpotent}.
It follows that $S$ is finitely generated.\par 
Now suppose that $S$ has no idempotent.  Let $a$ be a fixed element of $S.$  
The {\em Tamura order} on $S$ (with respect to $a$) is defined by $$x\leq_a y\iff x=a^ny\text{ for some }n\geq 0.$$
By \cite[Section IV.4]{Grillet}, there exists a set $M$ of maximal elements of $S$ (under $\leq_a$), where $a\in M,$
such that every element of $S$ can be written in the form $p_n=a^np$ with $n\geq 0$ and $p\in M,$ and the set $I=S\!\setminus\!M$ is an ideal.  
The Rees quotient $S/I$ is f-noetherian by Lemma \ref{quotient}.
We have that $S/I$ is nilpotent, since $S$ is an archimedean semigroup, and hence it is finite by Remark \ref{quotient remark} and Lemma \ref{nilpotent}.
Therefore, the set $M$ is finite and $S=\langle M\rangle$ is finitely generated.
\end{proof}

\begin{defn}
A commutative semigroup $S$ is said to be {\em complete} if every archimedean component of $S$ contains an idempotent.
\end{defn}

Our next result states that f-noetherian complete semigroups are finitely generated.

\begin{prop}
\label{complete}
Let $S$ be a complete semigroup.  If $S$ is f-noetherian, then it is finitely generated.
\end{prop}

\begin{proof}
We have that $S$ is a semilattice of archimedean semigroups $\mathcal{S}(Y, S_{\alpha})$ where each $S_{\alpha}$ contains an idempotent.
Note that $\mathcal{H}$ is a congruence on $S.$
For each $\alpha\in Y,$ we have that $T_{\alpha}=S_{\alpha}/\mathcal{H}_{\alpha}$ is a nilpotent semigroup. 
Therefore, if $T_{\alpha}$ is f-noetherian, then it is finite by Lemma \ref{nilpotent}.
It follows from Lemma \ref{congruence contained in L} that $S$ is finitely generated.
\end{proof}

\begin{remark}
There exist finitely connected complete semigroups that are not finitely generated; e.g. infinite finitely connected semilattices. 
\end{remark}

We now exhibit an example of an uncountable commutative semigroup that is f-noetherian.

\begin{ex}
Let $A=\langle a\rangle\cong\mathbb{N}.$ 
Let $B$ be the set of all infinite bounded sequences of non-negative integers, and define $xy=0$ for all $x, y\in B,$ where $0$ denotes the sequence whose every term is $0$;
note that $B$ is a null semigroup.
We denote a sequence $(x_1, x_2, \dots)$ by $(x_n).$
Let $S=A\cup B,$ and define a multiplication on $S,$ extending those on $A$ and $B,$ as follows:
$$a^i(x_n)=(x_n)a^i=(y_n),$$ where $y_n=\text{max}(x_n-i, 0)$ for each $n\in\mathbb{N}.$
It it can be easily seen that, with this multiplication, $S$ is an uncountable commutative semigroup.  We now show that $S$ is f-noetherian.\par
Notice that $(x_n)=a^i(x_n+i),$ and $a^i(x_n)=0$ if and only if $i\geq\text{max}\{x_n : n\in\mathbb{N}\}.$
Let $\rho$ be a congruence of finite index on $S.$  
Suppose first that there exists $i\in\mathbb{N}$ such that $a^i\,\rho\,0.$  
Let $n$ be minimal such that $a^n\,\rho\,0$; we then have that $\{a^i\}$ is a singleton $\rho$-class for every $i<n.$
We claim that $\rho$ is generated by the pair $(a^n, 0).$  
Indeed, for any $s\in\{a^i : i>n\}\cup B,$ there exists $t\in A\cup B$ such that $s=a^nt,$ so we obtain $(s, 0)$ by applying the pair $(a^n, 0),$ and the claim follows.\par 
Now suppose that there are no $i\in\mathbb{N}$ such that $a^i\,\rho\,0.$  
We have the $\rho|_A$ is generated by a pair $(a^m, a^n)$ with $m<n.$
We claim that $\rho$ is generated by $(a^m, a^n).$
Clealy it is enough to show that, for every $x\in B,$ the pair $(x, 0)$ is a consequence of $(a^m, a^n).$
So, let $x=(x_n)\in B$ and let $t=\text{max}\{x_n : n\in\mathbb{N}\}.$  There exists $q\in\mathbb{N}$ such that $t\leq qk,$ where $k=n-m.$
Since $x=a^m(x_n+m),$ $0=a^{m+qk}(x_n+m),$ and $(a^m, a^{m+qk})$ is a consequence of $(a^m, a^n),$ it follows that $(x, 0)$ is a consequence of $(a^m, a^n),$ as required.
\end{ex}

For cancellative commutative semigroups, being countable is a necessary condition for being f-noetherian.

\begin{prop}
Let $S$ be a cancellative commutative semigroup.  If $S$ is f-noetherian, then it is countable.
\end{prop}

\begin{proof}
By Theorem \ref{archimedean decomposition} we have that $S$ is a semilattice of archimedean semigroups $\mathcal{S}(Y, S_{\alpha}).$ 
Now $Y$ is finite by Corollary \ref{finite semilattice}.
We prove that $S$ is countable by induction on the order of $Y.$   
If $|Y|=1,$ then $S$ is finitely generated, and hence countable, by Proposition \ref{archimedean}.
Suppose that $|Y|>1,$ and choose a maximal element $\beta\in Y.$
We have that $S_{\beta}$ is f-noetherian by Lemma \ref{complement ideal}, and hence finitely generated by Proposition \ref{archimedean}.\par
Suppose that $S_{\beta}$ contains an idempotent.  
Consider $T=S/\mathcal{H}=\mathcal{S}(Y, T_{\alpha}),$ where $T_{\alpha}=S/\mathcal{H}_{\alpha}.$
From the proof of Proposition \ref{archimedean}, we have that $S_{\beta}$ is an ideal extension of a group by a finite nilpotent semigroup;
hence $T_{\beta}$ is finite.
Therefore, by Theorem \ref{large}, we have that $T\!\setminus\!T_{\beta}=\mathcal{S}(Y\!\setminus\!\{\beta\}, T_{\alpha})$ is f-noetherian. 
By induction we have that $T\!\setminus\!T_{\beta}$ is countable, and hence $T$ is countable, so $S$ has countably many $\mathcal{H}$-classes.
Now, for every $\mathcal{H}$-class $H,$ the Sch{\"u}tzenberger group $\Gamma(H)$ is finitely generated by Proposition \ref{Schutzenberger};
therefore, since $|H|=|\Gamma(H)|,$ we have that $H$ is countable.
Hence, we have that $S,$ being a countable union of countable sets, is countable.\par
Now suppose that $S_{\beta}$ has no idempotent, and fix an element $a\in S_{\beta}.$
Recall from the proof of Proposition \ref{archimedean} that there exists a finite set $M$ of maximal elements in $S_{\beta}$ under the Tamura order $\leq_a,$ where $a\in M,$
such that every element of $S_{\beta}$ can be written in the form $p_n=a^np$ with $n\geq 0$ and $p\in M.$
We now define a congruence $\rho$ on $S$ by
$$s\,\rho\,t\iff s, t\in S_{\alpha}\text{ for some }\alpha\in Y, a^ms=a^nt\text{ for some }m, n\geq 0.$$
Let $T=S/\rho.$  Letting $\rho_{\alpha}$ denote $\rho$ restricted to $S_{\alpha},$ we have that $T=\mathcal{S}(Y, T_{\alpha}),$ where $T_{\alpha}=S_{\alpha}/\rho_{\alpha}.$
It is clear that $T_{\beta}$ is finite.  The same argument as above proves that $T$ is countable, so $S$ has countably many $\rho$-classes.
Now consider a $\rho$-class $C,$ and fix $c\in C.$
Since $S$ is countable, for each $s\in S$ there exists some unique $n_s\geq 0$ such that $a^{n_s}s=c$ or $s=a^{n_s}c.$
We therefore have an injection $$C\to\mathbb{N}_0\times\mathbb{N}_0, s\mapsto
\begin{cases}
 (n_s, 0)\text{ if }a^{n_s}s=c,\\
 (0, n_s)\text{ if }s=a^{n_s}c.
\end{cases}$$
It follows that each $\rho$-class $C$ is countable, and hence $S$ is countable.
\end{proof}

\section{\large{Simple and $0$-simple semigroups}\nopunct}

In this section we consider simple and $0$-simple semigroups.
We have already seen in Example \ref{polycyclic} that the properties of being f-noetherian and being finitely generated do not coincide for the class of $0$-simple semigroups.
On the other hand, it was observed in Remark \ref{left simple} that f-noetherian left simple semigroups are finitely generated.
We shall show that this is not the case for the class of right simple semigroups by considering Baer-Levi semigroups.\par 
So, let $X$ be an infinite set, and let $S$ be the set of all injective mappings $\alpha : X\to X$ such that $|X\!\setminus\!X\alpha|$ is infinite.
Under composition of mappings $S$ is a semigroup, called the {\em Baer-Levi semigroup} on $X$.
It turns out that Baer-Levi semigroups are right cancellative, right simple and idempotent-free \cite[Theorem 8.2]{C&P}.

\begin{thm}
\label{Baer-Levi}
The Baer-Levi semigroup $S$ on an infinite set $X$ is f-noetherian. 
\end{thm}

\begin{proof}
Partition $X$ as $X=\bigcup_{i=1}^{\infty}X_i,$ where the $X_i$ are disjoint subsets of $X,$ each with the same cardinality as $X,$ and define bijections $\alpha_i : X\to X_i.$
We prove that, for any distint $i, j\in\mathbb{N},$ the pair $(\alpha_i, \alpha_j)$ generates the universal right congruence.
It then follows that the universal right congruence is the only right congruence of finite index on $S,$ and $S$ is f-noetherian.\par 
So, let $\rho$ be the right congruence on $S$ generated by the pair $(\alpha_i, \alpha_j).$  
We make the following claim.
\begin{claim*}
If $\theta, \phi\in S$ such that $|X\!\setminus\!(X\theta\cup X\phi)|=\infty,$ then $(\theta, \phi)\in\rho.$
\end{claim*}
\begin{proof}
Choose $\gamma\in S$ such that $\gamma|_{X_i}=\alpha_i^{-1}\theta$ and $(X\!\setminus\!X_i)\gamma\subseteq X\!\setminus\!(X\theta\cup X\phi),$
and choose $\delta\in S$ such that $\delta|_{X_i}=\alpha_i^{-1}\alpha_j\gamma$ and $\delta|_{X_j}=\alpha_j^{-1}\phi.$
We then have a sequence $$\theta=\alpha_i\gamma, \alpha_j\gamma=\alpha_i\delta, \alpha_j\delta=\phi,$$
so $(\theta, \phi)\in\rho.$
\end{proof}
Returning to the proof of Theorem \ref{Baer-Levi}, let $\theta, \phi\in S.$
If $|X\phi\!\setminus\!X\theta|<\infty,$ then we have that $|X\!\setminus\!(X\theta\cup X\phi)|=\infty,$ so $(\theta, \phi)\in\rho$ by the above claim.
Suppose now that $|X\phi\!\setminus\!X\theta|=\infty.$
Choose $\gamma\in S$ such that $\gamma|_{X_i}=\alpha_i^{-1}\theta$ and $(X\!\setminus\!X_i)\gamma\subseteq X\phi\!\setminus\!X\theta,$ and let $\theta^{\prime}=\alpha_j\gamma.$
We have that $\theta=\alpha_i\gamma,$ so $(\theta, \theta^{\prime})\in\rho.$
Since $X\theta^{\prime}\subseteq X\phi,$ we have that $|X\!\setminus\!(X\theta^{\prime}\cup X\phi)|=|X\!\setminus\!X\phi|=\infty,$
and hence $(\theta^{\prime}, \phi)\in\rho$ by the above claim.
It now follows by transitivity that $(\theta, \phi)\in\rho.$ 
Therefore, we have that $S\times S\subseteq\rho,$ and hence $\rho=S\times S,$ as required.
\end{proof}

\begin{remark}
\label{0-simple}~\par
\begin{enumerate}
\item Adjoining a zero to a Baer-Levi semigroup yields a right $0$-simple semigroup that is f-noetherian but not finitely generated.
Left $0$-simple f-noetherian semigroups, however, are finitely generated by Proposition \ref{L-classes}.
\item Finitely connected left $0$-simple semigroups are not in general finitely generated.
Indeed, any group with a zero adjoined is finitely connected.
\item Any $0$-simple right noetherian semigroup is finitely generated by \cite[Thm. 3.2]{Hotzel}.  
However, there exist finitely generated left/right $0$-simple semigroups that are not right noetherian; 
e.g. the monoid $G^0$ where $G$ is any finitely generated group with a non-finitely generated subgroup.
\end{enumerate}
\end{remark}

\begin{remark}
A semigroup $S$ is said to be {\em pseudo-finite} if there exists a finite set $X\subseteq S\times S$ and some $n\in\mathbb{N}$ such that for any $a, b\in S,$
there exists an $X$-sequence of length at most $n$ connecting $a$ and $b.$
Clearly this property is stronger than that of being finitely connected.
It was asked in \cite[Open Problem 8.10]{Dandan} whether every pseudo-finite semigroup contains a completely simple ideal.
It can be seen from the proof of Theorem \ref{Baer-Levi} that Baer-Levi semigroups are pseudo-finite, answering this question in the negative.
\end{remark}

It was shown in \cite[Corollary 3.6]{Dandan} that every finitely connected completely simple semigroup is finitely generated,
so we have the following generalisation of Lemma \ref{group}.

\begin{thm}
\label{completely simple}
The following are equivalent for a completely simple semigroup $S$:
\begin{enumerate}
\item $S$ is f-noetherian;
\item $S$ is finitely connected;
\item $S$ is finitely generated.
\end{enumerate}
\end{thm}

Finitely connected completely $0$-simple semigroups are not in general finitely generated by \cite[Corollary 6.10]{Dandan}.
However, the properties of being f-noetherian and being finitely generated coincide for completely $0$-simple semigroups.

\begin{thm}
\label{completely 0-simple}
Let $S$ be a completely $0$-simple semigroup.  Then $S$ is f-noetherian if and only if it is finitely generated.
\end{thm}

\begin{proof}
We only need to prove the direct implication.
Let $S=\mathcal{M}^0(G; I, J; P).$  Since $S$ is finitely connected, we have that $I$ is finite by \cite[Corollary 6.10]{Dandan}.
For each $j\in J,$ define a map $\theta_j : I\to\{1, 0\}$ by $$i\theta_j=
\begin{cases} 
   1 & \text{if }p_{ji}\in G\\
   0 & \text{if }p_{ji}=0.
\end{cases}$$
We define a right congruence $\rho$ on $S$ by setting $0\,\rho\,0$ and
$$(i_1, g_1, j_1)\,\rho\,(i_2, g_2, j_2)\iff\theta_{j_1}=\theta_{j_2}.$$
Since $I$ is finite, there are only finitely many maps of the form $\theta_j$, so $\rho$ has finite index.  
Therefore, since $S$ is f-noetherian, $\rho$ is generated by a finite set $X.$\par
We claim that the index set $J$ is finite, so that $S$ has finitely many $\mathcal{L}$-classes, and is hence finitely generated by Proposition \ref{L-classes}.
Indeed, suppose that $J$ is infinite, and let $J_0$ be the finite set of elements of $J$ that appear in $X.$ 
Choosing $i\in I$ and distinct elements $j, j^{\prime}\in J\!\setminus\!J_0$ such that $(i, 1_G, j)\,\rho\,(i, 1_G, j^{\prime}),$
we have an $X$-sequence
$$(i, 1_G, j)=x_1s_1, y_1s_1=x_2s_2, \dots, y_ks_k=(i, 1_G, j^{\prime}).$$
Since $\bigl((i, 1_G, j), 0\bigr)\notin\rho$ and $j\not\in J_0,$ it is easy to see, reading from left to right, 
that for each $i\in\{1, \dots, k\}$ the element $s_i$ is in $S\!\setminus\!\{0\}$ and its third coordinate is $j.$ 
But then $j=j^{\prime},$ so we have a contradiction, and hence $J$ is finite.
\end{proof}

\section{\large{Concluding remarks and open problems}\nopunct}

In this paper we have shown that the class of finitely generated semigroups is strictly contained within the class of f-noetherian semigroups.
Furthermore, we have investigated, for various classes of semigroups, whether the property of being f-noetherian is equivalent to being finitely generated, and the results are summarised in Table 1.
We have not been able to answer the following question.

\begin{prob}
Is every f-noetherian completely regular semigroup finitely generated?
\end{prob}

Building on the work of this paper, a natural progression would be to attempt to classify those f-noetherian semigroups that lie in certain important classes 
(beyond those for which being f-noetherian coincides with finite generation), such as commutative or inverse.
This appears to be a difficult task, particularly due to the presence of uncountable examples.
Cancellative commutative semigroups that are f-noetherian, however, were shown to be countable, so perhaps this would be a good starting point.\par
Another possible direction for future research would be to consider alternative conditions for a general semigroup or monoid to be f-noetherian.
It was shown in \cite{Dandan} that, for a monoid $M,$ 
the universal right congruence on $M$ being generated by a finite set $X$ is equivalent to the undirected left Cayley graph of $M$ with respect to $X$ being connected,
and is also equivalent to $M$ satisfying the homological finiteness property of being {\em type right-$FP_1$} \cite[Theorem 3.10]{Dandan}.
This leads us to ask whether the property of being f-noetherian for monoids can be similarly described in graph-theoretic or homological terms.

\section*{Acknowledgments}
The author would like to thank his supervisor, Professor Nik Ru{\v s}kuc, for his advice and guidance during the writing of this paper, and EPSRC for financial support.

\end{document}